\documentclass[a4paper,10pt]{article}
\usepackage[english,frenchb]{babel}
\usepackage[latin1]{inputenc}
\usepackage[T1]{fontenc}
\usepackage{amsmath}
\usepackage{array}
\usepackage{amsthm}
\usepackage{amssymb}
\input xy
\xyoption{all}

\newcommand{\A}{{\mathcal{A}}}
\newcommand{\B}{{\mathcal{B}}}
\newcommand{\C}{{\mathcal{C}}}

\newcommand{\Gr}{{\mathcal{G}r}}

\newcommand{\kk}{{\Bbbk}}

\title{Autour des résultats d'annulation cohomologique de Scorichenko}

\author{Aur\'elien DJAMENT\thanks{CNRS, laboratoire de math\'ematiques Jean Leray (Nantes) ; aurelien.djament@univ-nantes.fr}}

\date{Juillet 2009}

\newtheorem{thm-intro}{Théorème}

\newtheorem{thm}{Théorème}[section]
\newtheorem{pr}[thm]{Proposition}

\newtheorem{lm}[thm]{Lemme}

\theoremstyle{definition}
\newtheorem{defi}[thm]{Définition}
\newtheorem{nota}[thm]{Notation}

\theoremstyle{remark}
\newtheorem{rem}[thm]{Remarque}

\begin{document}

\maketitle

\begin{abstract}
 Le principal objectif de ces notes est de rendre disponible la démonstration de Scorichenko (\cite{Sco}), non publiée, de l'isomorphisme entre homologie des foncteurs et $K$-théorie stable à coefficients polynomiaux sur un anneau quelconque. Plus précisément, on donne la démonstration de l'étape clef (dans sa formulation {\em co}homologique), la plus originale, de l'argument de Scorichenko, qui consiste en un résultat d'annulation (co)homologique général extrêmement frappant. (Franjou et Pirashvili donnent dans  \cite{FP} tous les arguments de la démonstration de Scorichenko, hormis ce qui concerne le résultat d'annulation en question, qui n'y est établi que dans le cas particulier nettement plus simple des anneaux de dimension au plus~$1$.) Nous montrons aussi comment d'autres résultats d'annulation cohomologique connus peuvent être obtenus et généralisés par la même méthode.
\end{abstract}

Dans toute cette note, $\kk$ désigne un anneau commutatif de base fixé. 
On se donne également une catégorie additive $\A$ (essentiellement) petite. 

Si $\C$ est une catégorie (essentiellement) petite, on note $\C-\mathbf{Mod}$ la catégorie des foncteurs de $\C$ vers la catégorie, notée $\mathbf{Mod}$, des $\kk$-modules à gauche.
 (En fait, pour beaucoup de notre propos, la catégorie but $\mathbf{Mod}$ peut être remplacée par une catégorie abélienne assez gentille --- par exemple une catégorie de Grothendieck avec assez de projectifs.)

La catégorie $\C-\mathbf{Mod}$ est abélienne ; c'est une catégorie de Grothendieck (cf. \cite{Gab}, par exemple). Un ensemble de générateurs projectifs en est donné par les $P^\C_c=\kk[{\rm Hom}_\C(c,-)]$ ($\kk$-linéarisation du foncteur ensembliste ${\rm Hom}_\C(c,-)$), où $c$ parcourt les objets de $\C$ (ou un squelette) : on dispose d'un isomorphisme ${\rm Hom}_{\C-\mathbf{Mod}}(P^\C_c,F)\simeq F(c)$ naturel en $c\in {\rm Ob}\,\C$ et $F\in {\rm Ob}\,\C-\mathbf{Mod}$.

\section{Foncteurs polynomiaux}

Pour tout $d\in\mathbb{N}$, on note $t_d$ l'endofoncteur $A\mapsto A^{\oplus (d+1)}$ de $\A$ et $\tau_d$ l'endofoncteur de précomposition par $t_d$ de $\A-\mathbf{Mod}$. Si $I$ est une partie de l'ensemble $\mathbf{d}=\{0,1,\dots,d\}$, on note $u_I$ (resp. $p_I$) la transformation naturelle $id\to t_d$ (resp. $t_d\to id$) donnée par le morphisme $A\to A^{\oplus (d+1)}$ (resp. $A^{\oplus (d+1)}\to A$) dont la $i$-ème composante $A\to A$ est l'identité si $i-1\in I$ et $0$ sinon.

On définit les {\em effets croisés} de degré $d$ comme les transformations naturelles
$$cr_d=\sum_{I\subset\mathbf{d}}(-1)^{|I|}(u_I)_* : id\to\tau_d$$
(où $|I|$ désigne le cardinal de $I$) et
$$cr_d^{op}=\sum_{I\subset\mathbf{d}}(-1)^{|I|}(p_I)_* : \tau_d\to id.$$

On note que le foncteur $t_d$, donc également $\tau_d$, est {\em auto-adjoint}.

Un objet $F$ de $\A-\mathbf{Mod}$ est dit {\em polynomial} de degré au plus $d$ si $cr_d(F)=0$, condition équivalente à $cr_d^{op}(F)=0$ et impliquant $cr_i(F)=0$ pour $i\leq d$. (Cette condition est également équivalente au fait que $\Delta_X^{d+1}(F)=0$ pour tout $X$ dans un ensemble $E$ d'objets de $\A$ tel que tout objet de $\A$ soit facteur direct d'une somme directe finie d'éléments de $E$, où l'on note $\Delta_X$ le noyau de la transformation naturelle de la précomposition par $-\oplus X$ vers l'identité induite par la projection canonique $Y\oplus X\twoheadrightarrow Y$. Cette définition est usuellement employée, avec $E=\{A\}$, lorsque $\A$ est la catégorie des modules à gauche projectifs de type fini sur un anneau $A$, par exemple.)

\begin{nota}
On note $\kappa_d$ le conoyau de $cr_d$. 
\end{nota}

\section{Le critère d'annulation cohomologique de Scorichenko}

Toutes les définitions et propositions de cette section sont dues à Scorichenko (\cite{Sco}).

\begin{defi}
 Un foncteur $F\in {\rm Ob}\,\A-\mathbf{Mod}$ vérifie la condition $(AH)^d_n$ (où $d$ et $n$ sont des entiers naturels) si, pour tout $0\leq i\leq n$, le morphisme $cr_d(\kappa_d^i(F)) : \kappa_d^i(F)\to\tau_d \kappa_d^i(F)$ est injectif (où l'exposant $i$ indique la $i$-ème itération du foncteur $\kappa_d$). Si cette condition est vérifiée pour tout $n$, nous dirons que $F$ satisfait la condition $(AH)^d$.
\end{defi}

\begin{pr}\label{pran1}
 Supposons que $F$ est un foncteur vérifiant la condition $(AH)^d_n$ et $A$ un foncteur polynomial de degré au plus $d$. Alors ${\rm Ext}^i(A,F)=0$ pour $i\leq n$.
\end{pr}

\begin{proof}
On raisonne par récurrence sur $n$, supposant la conclusion vérifiée pour les foncteurs satisfaisant $(AH)^d_{n-1}$ (pour $n>0$). Dans la suite exacte longue de cohomologie associée à la suite exacte courte $0\to F\xrightarrow{cr_d(F)}\tau_d F\to\kappa_d F\to 0$, les morphismes
$$cr_d(F)_* :  {\rm Ext}^i(A,F)\to {\rm Ext}^i(A,\tau_d F)$$
sont nuls, puisqu'ils s'identifient par adjonction à
$$(cr^{op}_d(A))^* :  {\rm Ext}^i(A,F)\to {\rm Ext}^i(\tau_d A,F).$$

L'hypothèse de récurrence fournissant ${\rm Ext}^i(A,F)=0$ et ${\rm Ext}^i(A,\kappa_d F)=0$ pour $i<n$, on en déduit bien ${\rm Ext}^n(A,F)=0$ comme souhaité.
\end{proof}

On note $\mathbf{M}(\A)$ (resp. $\mathbf{M}_0(\A)$) la sous-catégorie de $\A$ dont les morphismes sont les monomorphismes scindés (resp. les morphismes nuls et les monomorphismes scindés) de $\A$.

Notons $\theta : \A-\mathbf{Mod}\to\mathbf{M}_0(\A)-\mathbf{Mod}$ le foncteur de restriction.

\begin{defi}
 On dit qu'un foncteur $F$ de $\A-\mathbf{Mod}$ vérifie la condition $(RM)^d$ si le morphisme $\theta(cr_d(F)) : \theta(F)\to\theta(\tau_d(F))$ de $\mathbf{M}_0(\A)-\mathbf{Mod}$ est un monomorphisme scindé.
\end{defi}

\begin{lm}\label{lmanf}
 La condition $(RM)^d$ implique la condition $(AH)^d$.
\end{lm}

\begin{proof}
 Comme le foncteur $\theta$ est fidèle, la condition $(RM)^d$ implique $(AH)^d_0$. Il suffit donc d'établir que si $F$ vérifie $(RM)^d$, alors il en est de même pour $\kappa_d F$. Cela sera une conséquence immédiate du lemme suivant.
\end{proof}

\begin{lm}\label{auxscg}
 \begin{enumerate}
  \item Il existe une transformation naturelle involutive $\gamma : \tau_d^2\to\tau_d^2$ telle que $\gamma\circ\tau_d(cr_d)=cr_d(\tau_d)$. 
  \item Disons qu'un morphisme $f : F\to G$, où $F$ et $G$ sont deux foncteurs vérifiant $(RM)^d$ pour lesquels on a choisi des rétractions $r_F : \theta\tau_d(F)\to\theta F$ et $r_G : \theta\tau_d(G)\to\theta G$ aux effets croisés, est $(RM)^d$-compatible si le diagramme
  $$\xymatrix{\theta\tau_d F\ar[rr]^-{\theta\tau_d(f)}\ar[d]^-{r_F} & & \theta\tau_d G\ar[d]^-{r_G}\\
  \theta F\ar[rr]^-{\theta(f)} & & \theta G
  }$$
  commute.
  
  Le noyau et le conoyau d'un tel morphisme vérifient la propriété $(RM)^d$.
  \item Si $F$ vérifie la propriété $(RM)^d$, il en est de même pour $\tau_d F$ ; de plus on peut choisir les rétractions de sorte que $cr_d(F) : F\to\tau_d F$ soit $(RM)^d$-compatible.
 \end{enumerate}
\end{lm}

\begin{proof} Le foncteur $\tau_d^2$ s'identifie à la précomposition par l'endofoncteur $A\mapsto A\underset{\mathbb{Z}}{\otimes}\mathbb{Z}^{\mathbf{d}}\otimes\mathbb{Z}^{\mathbf{d}}$
de $\A$. On vérifie facilement que la précomposition par l'involution intervertissant les deux derniers facteurs procure la transformation naturelle $\gamma$ souhaitée.

Le deuxième point est immédiat.

Pour la dernière assertion, on constate que le morphisme
$$\theta\tau_d^2 F\xrightarrow{\theta\gamma_F}\theta\tau_d^2 F\xrightarrow{\tau_d(r_F)}\theta F$$
(dans la dernière flèche, on a encore écrit, par abus, $\tau_d$ pour le foncteur analogue défini sur $\mathbf{M}_0(\A)-\mathbf{Mod}$) fournit une rétraction convenable à $\theta cr_d(\tau_d(F))$. La $(RM)^d$-compatibilité provient de ce que $cr_d$ est définie à partir de la précomposition par des morphismes appartenant à $\mathbf{M}_0(\A)$.
\end{proof}

\begin{pr}\label{pranf}
 Soient $F$ et $A$ des objets de $\A-\mathbf{Mod}$. On suppose que $F$ vérifie la condition $(RM)^d$ et que $A$ est polynomial de degré au plus $d$. Alors ${\rm Ext}^*(A,F)=0$.
\end{pr}

\begin{proof}
 C'est une conséquence immédiate du lemme~\ref{lmanf} et de la proposition~\ref{pran1}.
\end{proof}

\begin{rem}
 On a bien sûr un énoncé d'annulation analogue en homologie, en remplaçant les groupes d'extensions par des groupes de torsion. L'énoncé originel de Scorichenko en est une variante en termes de bifoncteurs et de leur homologie (de Hochschild).
\end{rem}

\begin{rem}\label{rq-cbut}
 Pour les résultats de cette section, la catégorie but $\mathbf{Mod}$ peut être une catégorie de Grothendieck avec assez de projectifs arbitraire.
\end{rem}

\begin{rem}
 Si l'on a ${\rm Ext}^*_{\A-\mathbf{Mod}}(A,F)=0$ pour tout foncteur polynomial $A$ (sans restriction de degré), alors la même annulation vaut lorsque $A$ est {\em analytique}, i.e. colimite (qu'on peut supposer filtrante, les foncteurs polynomiaux formant une sous-catégorie épaisse) de sous-foncteurs polynomiaux. Cela se déduit de la suite spectrale cohomologique associée aux colimites filtrantes.
 
 Cette annulation sur les foncteurs analytiques vaudra donc dans tous les exemples traités dans la suite de cette note.
\end{rem}

\section{Le théorème de comparaison homologique de Scorichenko}

Notons $\alpha : \A-\mathbf{Mod}\to\mathbf{M}(\A)-\mathbf{Mod}$ le foncteur de restriction. On donne ici la démonstration, due à Scorichenko, du résultat suivant :

\begin{thm}\label{thf-sco}
 Soient $A$ et $F$ deux $\A$-modules à gauche, $A$ étant supposé polynomial. Le morphisme naturel
 $${\rm Ext}^*_{\A-\mathbf{Mod}}(A,F)\to {\rm Ext}^*_{\mathbf{M}(\A)-\mathbf{Mod}}(\alpha(A),\alpha(F))$$
 induit par la restriction est un isomophisme.
\end{thm}

\begin{proof} On peut supposer que $A$ et $F$ sont nuls en $0$, puisque le foncteur constant $\kk=P^{\mathbf{M}(\A)}_0$ est projectif dans $\mathbf{M}(\A)-\mathbf{Mod}$.

 Le foncteur $\alpha$ possède un adjoint à droite $\beta : \mathbf{M}(\A)-\mathbf{Mod}\to\A-\mathbf{Mod}$, dont nous noterons $\mathbf{R}^*\beta$ les foncteurs dérivés à droite. L'adjonction entre le foncteur exact $\alpha$ et le foncteur $\beta$ se dérive en une suite spectrale convergente
 $$E^{p,q}_2={\rm Ext}^p_{\A-\mathbf{Mod}}(A,\mathbf{R}^q\beta(X))\Rightarrow {\rm Ext}^{p+q}_{\mathbf{M}(\A)-\mathbf{Mod}}(\alpha(A),X).$$
 Par conséquent, le théorème sera démontré si l'on établit que ${\rm Ext}^*_{\A-\mathbf{Mod}}(A,(\mathbf{R}^*\beta\alpha F)/F)=0$ (l'unité $F\to\beta\alpha F$ de l'adjonction est injective car le foncteur $\alpha$ est exact et fidèle).

\medskip

Par le lemme de Yoneda, le foncteur $\beta$ est donné par
$$\beta(X)(a)={\rm Hom}_{\mathbf{M}(\A)-\mathbf{Mod}}(\alpha P^\A_a,X)\;;\text{ plus généralement on a}$$
$$\mathbf{R}^i\beta(X)(a)={\rm Ext}^i_{\mathbf{M}(\A)-\mathbf{Mod}}(\alpha P^\A_a,X).$$
Noter que $\mathbf{R}^*\beta(X)$ est comme $X$, nul en $0$.

Soit $d\in\mathbb{N}$. On définit une transformation naturelle graduée
$$\xi : \theta\tau_d\mathbf{R}^*\beta\alpha(F)\to\theta\mathbf{R}^*\beta\alpha(F)$$
de la manière suivante : pour tout objet $a$ de $\A$, $\xi_a$ est le morphisme
$$\xi_a : {\rm Ext}^*_{\mathbf{M}(\A)-\mathbf{Mod}}(\alpha P^\A_{t_d(a)},\alpha F)\to {\rm Ext}^*_{\mathbf{M}(\A)-\mathbf{Mod}}(\alpha\sigma^d_a P^\A_{t_d(a)},\alpha\sigma^d_a F)\to {\rm Ext}^*_{\mathbf{M}(\A)-\mathbf{Mod}}(\alpha P^\A_a,\alpha F)$$
où la première flèche est induite par l'endofoncteur exact $_\mathbf{M}\sigma^d_a$ de $\mathbf{M}(\A)-\mathbf{Mod}$ donné par la précomposition par $-\oplus a^{\oplus d}$, qui vérifie la propriété $_\mathbf{M}\sigma^d_a\circ\alpha=\alpha\circ\sigma^d_a$, où $\sigma^d_a$ est la précomposition analogue sur $\A-\mathbf{Mod}$, et la seconde flèche est induite par la précomposition par le morphisme $P^\A_a=\mathbb{Z}[{\rm Hom}_\A[(a,-)]\to\sigma^d_a P^\A_{t_d(a)}=\mathbb{Z}[{\rm Hom}_\A(a\oplus a^{\oplus d},-\oplus a^{\oplus d})]$ induit par le foncteur $-\oplus a^{\oplus d}$ et la postcomposition par l'épimorphisme naturel scindé $\sigma^d_a F\twoheadrightarrow F$ (déduit des projections canoniques $b\oplus a^{\oplus d}\twoheadrightarrow b$).

\medskip 

Avant toute chose, remarquons que le diagramme
\begin{equation}\label{eqrel}
\xymatrix{\tau_d(F)(a)\ar[r]^{p_{\{0\}}(a)}\ar[d] & F(a)\ar[d] \\
 \tau_d\mathbf{R}^*\beta\alpha(F)(a)\ar[r]_{\xi_a} & \mathbf{R}^*\beta\alpha(F)(a)}
\end{equation}
dont les flèches verticales sont données par l'unité de l'adjonction
commute, parce que $p_{\{0\}}(a)$ peut se voir comme la composée
$$\tau_d F(a)\simeq {\rm Hom}_{\A-\mathbf{Mod}}(P^\A_{t_d(a)},F)\to {\rm Hom}_{\A-\mathbf{Mod}}(\sigma^d_a P^\A_{t_d(a)},\sigma^d_a F)\to {\rm Hom}_{\A-\mathbf{Mod}}(P^\A_a,F)\simeq F(a)$$
définie de manière analogue à $\xi_a$.

\medskip

Vérifions que les applications linéaires $\xi_a$ sont fonctorielles en $a\in{\rm Ob}\,\mathbf{M}_0(\A)$.

Si $f : a\to b$ est un monomorphisme scindé, $g : b\to a$ une rétraction de $f$ et $a'$ un objet de $\A$ tel que $b\simeq a\oplus a'$ et que $f$, $g$ s'identifient respectivement à l'injection $a\hookrightarrow a\oplus a'$ et à la projection standard $a\oplus a'\twoheadrightarrow a$ via cet isomorphisme, le diagramme suivant
$$\xymatrix{{\rm Ext}^*_{\mathbf{M}(\A)-\mathbf{Mod}}(\alpha P^\A_{t_d(a)},\alpha F)\ar[r]\ar[dd]_{f_*} & {\rm Ext}^*_{\mathbf{M}(\A)-\mathbf{Mod}}(\alpha\sigma^d_a P^\A_{t_d(a)},\alpha\sigma^d_a F)\ar[r]\ar[d]_{f_*} & {\rm Ext}^*_{\mathbf{M}(\A)-\mathbf{Mod}}(\alpha P^\A_a,\alpha F)\ar[dd]^{f_*} \\
& {\rm Ext}^*_{\mathbf{M}(\A)-\mathbf{Mod}}(\alpha\sigma^d_a P^\A_{t_d(b)},\alpha\sigma^d_a F)\ar[d]^l\ar[dr]^k & \\
 {\rm Ext}^*_{\mathbf{M}(\A)-\mathbf{Mod}}(\alpha P^\A_{t_d(b)},\alpha F)\ar[r]\ar[ur]^j & {\rm Ext}^*_{\mathbf{M}(\A)-\mathbf{Mod}}(\alpha\sigma^d_b P^\A_{t_d(b)},\alpha\sigma^d_b F)\ar[r] & {\rm Ext}^*_{\mathbf{M}(\A)-\mathbf{Mod}}(\alpha P^\A_b,\alpha F)
}$$
commute, où les lignes horizontales sont les flèches définissant $\xi_a$ et $\xi_b$, $j$ est induit par le foncteur exact $_\mathbf{M}\sigma^d_a$, $k$ par les morphisme $\sigma^d_a F\twoheadrightarrow F$ et $P^\A_a\to\sigma^d_a P^\A_{t_d(a)}$ décrits plus haut et $l$ est induit par le morphisme $_\mathbf{M}\sigma^d_{a'}$ (utiliser l'isomorphisme $\sigma_b\simeq\sigma_{a'}\circ\sigma_a$ déduit de l'isomorphisme $b\simeq a\oplus a'$ et l'analogue avec $_\mathbf{M}\sigma$). Cela établit la naturalité de $\xi$ relativement à $f$.

Il ne reste donc plus, pour voir que $\xi$ définit une transformation naturelle $\theta\circ\tau_d(\mathbf{R}^*\beta\alpha(F))\to\theta(\mathbf{R}^*\beta\alpha(F))$, qu'à constater la naturalité de $\xi$ relativement aux morphismes nuls, qui découle de la nullité en $0$ de tous les foncteurs considérés.

\medskip

Examinons maintenant la composition $\xi\circ\theta(cr_d(\mathbf{R}^*\beta\alpha(F)))$. Pour cela, on considère, pour $I\subset\mathbf{d}$, le diagramme commutatif
$$\xymatrix{{\rm Ext}^*(\alpha P^\A_a,\alpha F)\ar[r]^{(u_I)_*}\ar[d] & {\rm Ext}^*(\alpha P^\A_{t_d(a)},\alpha F)\ar[d] & \\
{\rm Ext}^*(\alpha\sigma^d_a P^\A_a,\alpha\sigma^d_a F)\ar[r]_{(u_I)_*} & {\rm Ext}^*(\alpha\sigma^d_a P^\A_{t_d(a)},\alpha\sigma^d_a F)\ar[r] & {\rm Ext}^*(\alpha P^\A_a,\alpha F)
}$$
(où les groupes d'extensions sont pris dans la catégorie $\mathbf{M}(A)-\mathbf{Mod}$ et les flèches verticales sont induites par le foncteur exact $_\mathbf{M}\sigma^d_a$) qui explicite la composition
\begin{equation}\label{eqcomp}
\mathbf{R}^*\beta\alpha(F)(a)\xrightarrow{(u_I)_*}\tau_d\mathbf{R}^*\beta\alpha(F)(a)\xrightarrow{\xi_a}\mathbf{R}^*\beta\alpha(F)(a).
\end{equation}

Sa ligne inférieure est induite par la postcomposition par l'épimorphisme canonique $\alpha\sigma^d_a F\twoheadrightarrow\alpha F$ et la précomposition par le morphisme composé
\begin{equation}\label{compp}
\alpha P^\A_a\to\alpha\sigma^d_a P^\A_{t_d(a)}\to\alpha\sigma^d_a P^\A_a 
\end{equation}
$$[f]\mapsto [f\oplus a^{\oplus d}]\mapsto [(f\oplus a^{\oplus d})\circ u_I]=[f\circ u_0,u_1,\dots,u_d]\qquad (\text{pour }f\in {\rm Hom}_\A(a,b))$$
(où l'on a écrit $u_I=(u_0,u_1,\dots,u_d)$).

On distingue ensuite les cas suivants.
\begin{enumerate}
 \item $I=\{0\}$. Notre flèche est alors induite par l'injection canonique de foncteurs $id\hookrightarrow\sigma^d_a$, qui provient elle-même du morphisme canonique de foncteur $id\to_\mathbf{M}\sigma^d_a$ (données par l'injection canonique $b\hookrightarrow b\oplus a^{\oplus d}$). Par conséquent, la composition
 $${\rm Ext}^*(\alpha P^\A_a,\alpha F)\to {\rm Ext}^*(\alpha\sigma^d_a P^\A_a,\alpha\sigma^d_a F)\to {\rm Ext}^*(\alpha P^\A_a,\alpha\sigma^d_a F)$$
 coïncide avec la postcomposition par le morphisme naturel $\alpha F\to\alpha\sigma^d_a F$ induit par $F\hookrightarrow\sigma^d_a F$. Comme la postcomposition de morphisme par l'épimorphisme canonique $\sigma^d_a F\twoheadrightarrow\alpha F$ est l'identité, on voit que la composée (\ref{eqcomp}) égale l'identité dans le cas qu'on considère.
 \item $I=\varnothing$. Le morphisme~(\ref{eqcomp}) est alors nul, puisque $\mathbf{R}^*\beta\alpha(F)$ est nul en~$0$.
 \item $I\notin\{\{0\},\varnothing\}$. Le morphisme $(f\circ u_0,u_1,\dots,u_d)$ est alors toujours un monomorphisme scindé (il exite $i\geq 1$ tel que $u_i=id_a$) ; autrement dit, le morphisme (\ref{compp}) se factorise par l'inclusion de $_\mathbf{M}\sigma^d_a P^{\mathbf{M}(\A)}_a$ dans $\alpha\sigma^d_a P^\A_a$. On peut par conséquent former un diagramme commutatif 
$$\xymatrix{{\rm Ext}^*(\alpha P^\A_a,\alpha F)\ar[r]\ar[d] & {\rm Ext}^*(\alpha\sigma^d_a P^\A_a,\alpha\sigma^d_a F)\ar[r]\ar[d] & {\rm Ext}^*(\alpha P^\A_a,\alpha F) \\
{\rm Ext}^*(P^{\mathbf{M}(\A)}_a,\alpha F)\ar[r] & {\rm Ext}^*(_\mathbf{M}\sigma^d_a P^{\mathbf{M}(\A)}_a,\alpha\sigma^d_a F)\ar[ru] &
}$$
dont la ligne supérieure a pour composée (\ref{eqcomp}) et dont les deux flèches verticales sont induites par le foncteur exact $_\mathbf{M}\sigma^d_a$. On en déduit facilement que (\ref{eqcomp}) coïncide avec 
$$\mathbf{R}^*\beta\alpha(F)(a)\to F(a)\xrightarrow{F(u_0)}F(a)\hookrightarrow\mathbf{R}^*\beta\alpha F(a).$$ 
\end{enumerate}

Cela montre que le morphisme $\theta\tau_d(\mathbf{R}^*\beta\alpha(F)/F)\to\theta(\mathbf{R}^*\beta\alpha(F)/F)$ induit par $\xi$ (le diagramme commutatif (\ref{eqrel}) en justifie l'existence) est une rétraction de l'effet croisé $\theta cr_d(\mathbf{R}^*\beta\alpha(F)/F)$. La proposition~\ref{pranf} permet de conclure.
\end{proof}

\begin{rem} On a un énoncé analogue en termes de groupes de torsion ou d'homologie de Hochschild de bifoncteurs.
\end{rem}

\section{Autres résultats d'annulation cohomologique déduits du critère général}

On commence par donner une généralisation d'un résultat classique dû à Franjou, qui traite le cas où $\A$ est la catégorie des espaces vectoriels de dimension finie sur un corps fini.

\begin{thm}\label{th-fra}
 Soient $a$ un objet de $\A$ et $\bar{P}^\A_a=Ker\,(P^\A_a\to P^\A_0)$ le projectif réduit de $\A-\mathbf{Mod}$ associé à $a$. Alors
 $${\rm Ext}^*_{\A-\mathbf{Mod}}(F,\bar{P}^\A_a)=0\quad\text{si }F\in {\rm Ob}\,\A-\mathbf{Mod}\text{ est polynomial.}$$
\end{thm}

\begin{proof}
 Considérons le foncteur $\overline{\mathbb{Z}[-]} : \mathbf{Ab}\to\mathbf{Ab}$ noyau de l'idéal d'augmentation $\mathbb{Z}[-]\to\mathbb{Z}$. Pour tout $d\in\mathbb{N}$, on définit un morphisme $\theta\tau_d\overline{\mathbb{Z}[-]}\to\theta\overline{\mathbb{Z}[-]}$ en envoyant $[x_0,\dots,x_d]-[0]$ sur $[x_0]-[0]$ si les $x_i$ sont non nuls pour $i>0$ et sur $0$ sinon. C'est une rétraction de $\theta cr_d(\overline{\mathbb{Z}[-]})$, de sorte que $\overline{\mathbb{Z}[-]}$ vérifie la propriété $(RM)^d$. Mais cette propriété étant stable par précomposition par un foncteur additif, on en déduit que  $\bar{P}^\A_a=\overline{\mathbb{Z}[-]}\circ {\rm Hom}_\A(a,-)$ la satisfait également, de sorte que la proposition~\ref{pranf} donne le résultat.
\end{proof}

\bigskip

Supposons maintenant donné un foncteur $T$ de $\A$ vers la catégorie $\mathbf{Ens}_*$ des ensembles pointés. Notons $\A_T$ la catégorie des couples $(a,t)$, où $a$ est un objet de $\A$ et $t$ un élément de $T(a)$, les morphismes de $(a,t)$ vers $(b,u)$ étant les morphismes $f : a\to b$ de $\A$ tels que $T(f)(t)=u$. Le foncteur $\A_T\to\A\quad (a,t)\mapsto a$ induit un foncteur $\iota_T : \A-\mathbf{Mod}\to\A_T-\mathbf{Mod}$ qui possède un adjoint à gauche exact $\omega_T$ tel que
$$\omega_T(X)(a)=\bigoplus_{t\in T(a)} X(a,t).$$
On définit un autre foncteur exact $\lambda_T : \A_T-\mathbf{Mod}\to\A-\mathbf{Mod}$ comme la précomposition par le foncteur $\A\to\A_T\quad a\mapsto (a,*)$ (où l'on note $*$ le point de base de $T(a)$). On dispose d'une transformation naturelle injective évidente $\lambda_T\hookrightarrow\omega_T$.

\begin{thm}\label{th-djagl}
 Sous l'hypothèse
 $$(H)\qquad\text{le morphisme naturel }\; T(A)\vee T(B)\to T(A\oplus B)\; \text{ est injectif}$$
(où $\vee$ désigne la somme dans la catégorie $\mathbf{Ens}_*$), l'inclusion naturelle $\lambda_T(X)\hookrightarrow\omega_T(X)$ induit pour tout $X\in {\rm Ob}\,\A_T-\mathbf{Mod}$ et tout $F\in {\rm Ob}\,\A-\mathbf{Mod}$ polynomial un isomorphisme
$${\rm Ext}^*_{\A-\mathbf{Mod}}(F,\lambda_T(X))\simeq {\rm Ext}^*_{\A-\mathbf{Mod}}(F,\omega_T(X)).$$
\end{thm}

\begin{proof}
 En remplaçant $X$ par son quotient par le sous-objet $X'$ défini par $X'(a,t)=X(a,*)$ si $t=*$, $0$ sinon, on se ramène au cas où $\lambda_T(X)=0$.
 
Soit $d\in\mathbb{N}$. On définit un morphisme $\xi : \theta\tau_d\omega_T(X)\to\theta\omega_T(X)$ comme suit : pour tout objet $a$ de $\A$,
$$\xi_a : \bigoplus_{u\in T(a^{\oplus (d+1)})} X(a^{\oplus (d+1)},u)\to\bigoplus_{t\in T(a)} X(a,t)$$
a pour composante $X(a^{\oplus (d+1)},u)\to X(a,u)$ le morphisme induit par $p_{\{0\}} : a^{\oplus (d+1)}\to a$ (projection sur le premier facteur) si $T(i_{\{0\}}) : T(a)\to T(a^{\oplus (d+1)})$ envoie $t$ sur $u$ (ce qui entraîne $T(p_{\{0\}})(u)=t$ puisque $p_{\{0\}}\circ i_{\{0\}}=id_a$), $0$ sinon. Le fait que $\xi$ est bien une transformation naturelle de foncteurs depuis $\mathbf{M}_0(\A)$ provient de ce que $\lambda_T(X)=0$ (qui assure la naturalité relativement aux morphismes nuls) et de l'hypothèse $(H)$ (qui implique que $T$ transforme un monomorphisme scindé en une injection).

La composée
$$\theta\omega_T(X)\xrightarrow{(u_I)_*}\theta\tau_d\omega_T(X)\xrightarrow{\xi}\theta\omega_T(X)$$
est nulle pour $I\neq\{0\}$, car l'hypothèse $(H)$ implique que les images de $T(u_I) : T(a)\to T(a^{\oplus (d+1)})$ et $T(u_{\{0\}})$ ne se rencontrent qu'en le point de base, de sorte que la nullité de $\lambda_T(X)$ permet de conclure.

Pour $I=\{0\}$, cette composée égale l'identité. Le foncteur $\omega_T(X)$ vérifie donc l'hypothèse $(RM)^d$ (pour tout $d$), d'où la conclusion par la proposition~\ref{pranf}.
\end{proof}

Un cas particulier important est celui où $\A$ est une catégorie abélienne (ou plus généralement une catégorie additive où tout morphisme se factorise, de manière unique à isomorphisme près, comme un épimorphisme suivi d'un monomorphisme, ce qui permet de disposer de la notion d'image) et le foncteur $T=\Gr$ associe à un objet $a$ de $\A$ l'ensemble de ses sous-objets (i.e. le quotient de l'ensemble des monomorphismes de but $a$ par l'action à gauche tautologique du groupe ${\rm Aut}_\A(a)$), pointé par le sous-objet nul. La propriété $(H)$ est clairement vérifiée. Mais on peut obtenir en fait mieux que le résultat donné par le théorème précédent. On introduit quelques notations supplémentaires à cet effet.

Soit $\mathbb{E}(\A)$ la sous-catégorie des épimorphismes (non nécessairement scindés) de $\A$. On dispose de foncteurs $r : \A_\Gr\to\mathbb{E}(\A)\quad (a,t)\mapsto t$ et $s : \A\times\mathbb{E}(\A)\to\A_\Gr\quad (a,x)\mapsto (a\oplus x,x)$. On note $\sigma : \A_\Gr-\mathbf{Mod}\to\A\times\mathbb{E}(\A)-\mathbf{Mod}\simeq\mathbf{Fct}(\A,\mathbb{E}(\A)-\mathbf{Mod})$ (catégorie des foncteurs de $\A$ vers $\mathbb{E}(\A)-\mathbf{Mod}$) la précomposition par $s$. Nous dirons qu'un foncteur $X$ de $\A_\Gr-\mathbf{Mod}$ est {\em polynomial} si $\sigma(X)$ l'est (la définition d'un foncteur polynomial dans $\mathbf{Fct}(\A,\mathbb{E}(\A)-\mathbf{Mod})$ est la même que dans $\A-\mathbf{Mod}=\mathbf{Fct}(\A,\mathbf{Mod})$). On définit enfin un endofoncteur $\nu$ de $\A_\Gr$ par
$$\nu(X)(a,t)=\bigoplus_{u\in\Gr(t)}X(a,u),$$
où un morphisme $f : (a,t)\to (a',t')$ induit $\nu(X)(f)$ ayant pour composante $X(a,u)\to X(a',u')$ le morphisme induit par $f$ si $f(u)=u'$, $0$ sinon. On remarque que $\nu$ est un sous-foncteur de $\iota\omega$ (on omet les indices $\Gr$) --- noter que $\iota\omega(X)(a,t)=\bigoplus_{u\in\Gr(a)}X(a,u)$.

\begin{thm}\label{th-dja2}
 Sous les hypothèses précédentes, si $X$ et $Y$ sont deux objets de $\A_\Gr-\mathbf{Mod}$, $X$ étant supposé polynomial, le morphisme naturel
 $${\rm Ext}^*_{\A_\Gr-\mathbf{Mod}}(X,\nu(Y))\to {\rm Ext}^*_{\A_\Gr-\mathbf{Mod}}(X,\iota\omega(Y))\xrightarrow{\simeq} {\rm Ext}^*_{\A-\mathbf{Mod}}(\omega(X),\omega(Y))$$
(où la première flèche est induite par l'inclusion $\nu(Y)\hookrightarrow\iota\omega(Y)$ et la seconde est l'isomorphisme déduit de l'adjonction entre les foncteurs exacts $\iota$ et $\omega$) est un isomorphisme.
\end{thm}

\begin{proof}
 Montrons d'abord que le morphisme naturel
 $${\rm Ext}^*_\B(F,\sigma\nu(Y))\to {\rm Ext}^*_\B(F,\sigma\iota\omega(Y))$$
 est un isomorphisme lorsque $F\in {\rm Ob}\,\B$ est polynomial, où l'on a posé $\B=\mathbf{Fct}(\A,\mathbb{E}(\A)-\mathbf{Mod})$.
 
 Pour cela, on définit $Z\in {\rm Ob}\,\mathbf{Fct}(\A_\Gr,\mathbb{E}(\A)-\mathbf{Mod})$ par
 $$Z(a,t)(u)=\underset{Im\,(v\hookrightarrow a\oplus u\twoheadrightarrow a)=t}{\bigoplus_{v\in\Gr(a\oplus u)}}Y(a\oplus u,v),$$
 l'effet sur les morphismes étant défini comme pour $\nu$. On constate que $\lambda(Z)=\sigma\nu(Y)$ tandis que $\omega(Z)=\sigma\iota\omega(Y)$, de sorte que le théorème~\ref{th-djagl} implique notre assertion. Ici on a encore noté $\omega$ et $\lambda$ les variantes à $\mathbf{Fct}(\A_\Gr,\mathbb{E}(\A)-\mathbf{Mod})$ des foncteurs initialement définis sur $\mathbf{Fct}(\A_\Gr,\mathbf{Mod})$, ce qui n'affecte pas la validité du théorème~\ref{th-djagl} en vertu de la remarque~\ref{rq-cbut} (on peut aussi déduire directement le résultat dont on a besoin du théorème~\ref{th-djagl} sous sa forme minimale par un argument standard de suite spectrale).
 
 Le foncteur $s$ est adjoint à gauche au foncteur $m : \A_\Gr\to\A\times\mathbb{E}(\A)\quad (a,t)\mapsto (a,t)$, de sorte que le foncteur $\sigma$ est adjoint à droite à la précomposition $\mu$ par $m$. Par conséquent, ce que nous venons d'établir montre que
 $${\rm Ext}^*_{\A_\Gr-\mathbf{Mod}}(X,\nu(Y))\to {\rm Ext}^*_{\A_\Gr-\mathbf{Mod}}(X,\iota\omega(Y))$$
 est un isomorphisme pour $X=\mu(F)$, où $F\in {\rm Ob}\,\B$ est polynomial.
 
 On utilise alors les observations suivantes :
 \begin{enumerate}
 \item les foncteurs $\mu$ et $\sigma$ préservent les foncteurs polynomiaux ; 
 \item la coünité $\mu\sigma(X)\to X$ de l'adjonction est surjective (elle provient en effet des morphismes $(a\oplus t,t)\to (a,t)$ de $\A_\Gr$ de composantes l'identité de $a$ et l'inclusion de $t$ dans $a$ ; ce sont des épimorphismes {\em scindés}), ce qui permet de construire une résolution simpliciale exacte du type
 $$\cdots\to (\mu\sigma)^n(X)\to\dots\to(\mu\sigma)^2(X)\to\mu\sigma(X)\to X\to 0$$
 (cf. \cite{Dja}, §\,7.2 pour plus de détail sur cette résolution).
\end{enumerate}
La comparaison des suites spectrales associées par application de ${\rm Ext}^*(-,\nu(Y))$ et ${\rm Ext}^*(-,\iota\omega(Y))$ permet donc de déduire le cas général requis du cas où $X$ est dans l'image du foncteur $\mu$ précédemment traité.
\end{proof}

Le cas où $\A$ est la catégorie des espaces vectoriels de dimension finie sur un corps fini est le théorème~10.2.1 de \cite{Dja}.
\bibliographystyle{smfalpha}
\bibliography{biblisco}
\end{document}